\documentclass[12p]{amsart}
\usepackage{amssymb}
\usepackage{amsmath}
\usepackage{amsfonts}
\usepackage{geometry}
\usepackage{graphicx}
\usepackage{mathrsfs,amssymb}

\usepackage{hyperref}
\usepackage{cleveref}

\theoremstyle{plain}

\newtheorem{lemma}{Lemma}

\newtheorem{remark}{Remark}

\newtheorem{theorem}{Theorem}
\numberwithin{equation}{section}

\begin{document}

\title{Scattering for the defocusing, cubic nonlinear Schr{\"o}dinger equation with initial data in a critical space}

\author{Benjamin Dodson}

\begin{abstract}
In this note we prove scattering for a defocusing nonlinear Schr{\"o}dinger equation with initial data lying in a critical Besov space. In addition, we obtain polynomial bounds on the scattering size as a function of the critical Besov norm.
\end{abstract}
\maketitle

\section{Introduction}
The qualitative long time behavior for the defocusing, nonlinear Schr{\"o}dinger equation
\begin{equation}\label{1.1}
i u_{t} + \Delta u = |u|^{p - 1} u, \qquad u(0,x) = u_{0},
\end{equation}
is completely worked out in the mass-critical ($p = \frac{4}{d} + 1$) and energy-critical ($p = \frac{4}{d - 2} + 1$) cases. Indeed, $(\ref{1.1})$ is known to be globally well-posed and scattering for any initial data $u_{0} \in \dot{H}^{s_{c}}(\mathbb{R}^{d})$,
\begin{equation}\label{1.2}
s_{c} = \frac{d}{2} - \frac{2}{p - 1},
\end{equation}
in the mass-critical (\cite{dodson2012global}, \cite{dodson2016global}, \cite{dodson2016global2}, \cite{MR3930589}, \cite{tao2007global}, \cite{killip2009cubic}, \cite{killip2009mass}) and energy-critical cases (\cite{bourgain1999global}, \cite{colliander2008global}, \cite{ryckman2007global}, \cite{visan2007defocusing}, \cite{killip2013nonlinear}, \cite{MR2154347}). The critical exponent $(\ref{1.2})$ arises from the fact that if $u(t,x)$ solves $(\ref{1.1})$, then for any $\lambda > 0$,
\begin{equation}\label{1.9}
u(t,x) \mapsto \lambda^{\frac{2}{p - 1}} u(\lambda^{2} t, \lambda x),
\end{equation}
also solves $(\ref{1.1})$, and the $\dot{H}^{s_{c}}$ norm of the initial data is invariant under $(\ref{1.9})$. On the other hand, well-posedness fails for $s < s_{c}$, see \cite{christ2003asymptotics}.

\begin{remark}
In this paper, global well-posedness refers to the existence of a global strong solution, i.e., a solution that satisfies Duhamel's principle
\begin{equation}\label{1.3}
u(t) = e^{it \Delta} u_{0} - i \int_{0}^{t} e^{i(t - \tau) \Delta} |u(\tau)|^{p - 1} u(\tau) d\tau,
\end{equation}
that is continuous in time, and depends continuously on the initial data. Scattering refers to the existence of $u_{+}$, $u_{-} \in \dot{H}^{s_{c}}(\mathbb{R}^{d})$ such that
\begin{equation}\label{1.4}
\lim_{t \nearrow \infty} \| u(t) - e^{it \Delta} u_{+} \|_{\dot{H}^{s_{c}}(\mathbb{R}^{d})} = 0,
\end{equation}
and
\begin{equation}\label{1.5}
\lim_{t \searrow -\infty} \| u(t) - e^{it \Delta} u_{-} \|_{\dot{H}^{s_{c}}(\mathbb{R}^{d})} = 0.
\end{equation}
See Chapter three of \cite{tao2006nonlinear} for a detailed treatment of global well-posedness and scattering for dispersive partial differential equations in general.
\end{remark}

The case when $p = \frac{4}{d} + 1$, ($s_{c} = 0$) is called mass-critical because a solution to $(\ref{1.1})$ preserves the mass, or $L^{2}$ norm of a solution,
\begin{equation}\label{1.7}
M(u(t)) = \int |u(t,x)|^{2} dx = M(u(0)).
\end{equation}
Likewise, the case when $p = \frac{4}{d - 2} + 1$, ($s_{c} = 1$) is called energy critical because a solution to $(\ref{1.1})$ preserves the energy,
\begin{equation}\label{1.8}
E(u(t)) = \int [\frac{1}{2} |\nabla u(t,x)|^{2} + \frac{1}{p + 1} |u(t,x)|^{p + 1}] dx = E(u(0)).
\end{equation}
The conserved quantities $(\ref{1.7})$ and $(\ref{1.8})$ imply that the $\dot{H}^{s_{c}}(\mathbb{R}^{d})$ norm is uniformly bounded for the entire time of existence of the solution to $(\ref{1.1})$ in the mass-critical and energy-critical cases. Thus, in the mass-critical and energy-critical cases, the proof of global well-posedness and scattering reduces to proving global well-posedness and scattering for a solution to $(\ref{1.1})$ with uniformly bounded $\dot{H}^{s_{c}}(\mathbb{R}^{d})$, which has been done.
\begin{remark}
In the energy-critical case, the Sobolev embedding theorem implies that $E(u(0)) < \infty$ when $u_{0} \in \dot{H}^{1}(\mathbb{R}^{d})$.
\end{remark}

It is conjectured that global well-posedness and scattering also hold for $(\ref{1.1})$ when $0 < s_{c} < 1$. In this case, there is no known conserved quantity that gives uniform bounds on the $\dot{H}^{s_{c}}(\mathbb{R}^{d})$ norm of a solution to $(\ref{1.1})$. Therefore, there are two possible ways in which a solution to $(\ref{1.1})$ might fail to scatter. These are called type one blowup and type two blowup. A solution to $(\ref{1.1})$ is called a type one blowup solution to $(\ref{1.1})$ if the $\dot{H}^{s_{c}}(\mathbb{R}^{d})$ norm is not uniformly bounded. Since $e^{it \Delta}$ is a unitary operator, an unbounded $\dot{H}^{s_{c}}(\mathbb{R}^{d})$ norm automatically precludes $(\ref{1.4})$ or $(\ref{1.5})$ from occurring. A blowup solution to $(\ref{1.1})$ is called a type two blowup solution if the solution fails to scatter but the $\dot{H}^{s_{c}}(\mathbb{R}^{d})$ norm is uniformly bounded for the entire time of its existence.

Type one blowup is known to occur for solutions to $(\ref{1.1})$ for some $d$ and $s_{c} > 1$, see \cite{merle2019blow}. Interestingly, the solutions obtained in  \cite{merle2019blow} have good regularity and good decay. By comparison, when $0 < s_{c} < 1$, if $u_{0} \in H_{x}^{1}(\mathbb{R}^{d})$, where $H_{x}^{1}$ is an inhomogeneous Sobolev space, then $(\ref{1.7})$, $(\ref{1.8})$, and interpolation imply a uniform bound on the $\dot{H}^{s_{c}}(\mathbb{R}^{d})$ norm when $0 < s_{c} < 1$. In fact, global well-posedness and scattering is known for a solution to $(\ref{1.1})$ when $u_{0} \in H_{x}^{1}(\mathbb{R}^{d})$ and $0 < s_{c} < 1$, see \cite{MR876933} and \cite{MR839728}.

Type two blowup has been precluded in many cases for $(\ref{1.1})$ when $0 < s_{c} < 1$. One particularly important case is the cubic nonlinear Schr{\"o}dinger equation in three dimensions (see \cite{erdHos2007derivation}),
\begin{equation}\label{1.6}
i u_{t} + \Delta u = |u|^{2} u, \qquad u : \mathbb{R} \times \mathbb{R}^{3} \rightarrow \mathbb{C}, \qquad u(0,x) = u_{0}.
\end{equation}
In this case, $s_{c} = \frac{3}{2} - \frac{2}{2} = \frac{1}{2}$. The main obstacle to proving scattering for $(\ref{1.6})$ with generic initial data in $\dot{H}^{1/2}(\mathbb{R}^{3})$ is the absence of a conservation law that controls the $\dot{H}^{1/2}$ norm of a solution to $(\ref{1.6})$ with initial data in $\dot{H}^{1/2}$. Observe that the momentum
\begin{equation}\label{1.9.1}
P(u(t)) = \int Im[\bar{u}(t,x) \nabla u(t,x)] dx,
\end{equation}
is conserved and scales like the $\dot{H}^{1/2}$ norm, but does not control the $\dot{H}^{1/2}$ norm of a solution to $(\ref{1.1})$.

In a breakthrough result, \cite{kenig2010scattering} proved that any solution to $(\ref{1.6})$ with $\| u(t) \|_{\dot{H}^{1/2}(\mathbb{R}^{3})}$ uniformly bounded on its entire interval of existence must be globally well-posed and scattering. Indeed, \cite{kenig2010scattering} and \cite{colliander2008global} were key in developing the concentration compactness method for the nonlinear Schr{\"o}dinger equation. Type two blowup was later precluded for a great many cases of $(\ref{1.1})$ when $0 < s_{c} < 1$, see \cite{murphy2015radial}, \cite{murphy2014intercritical}, and \cite{murphy2014defocusing}. Since the mass-critical and energy-critical problems reduce to type two blowup questions, the same techniques are useful for both problems.

In this paper we prove scattering for $(\ref{1.1})$ when $0 < s_{c} < 1$, $1 < p \leq 3$, and the initial data is radially symmetric and in the critical Besov space $B_{1,1}^{\frac{d}{2} + s_{c}}(\mathbb{R}^{d})$.
\begin{theorem}\label{t1.1}
The initial value problem $(\ref{1.1})$ is globally well-posed and scattering for radially symmetric initial data in the Besov space $B_{1,1}^{\frac{d}{2} + s_{c}}(\mathbb{R}^{d})$. In addition, when $1 < p < 3$, the scattering size,
\begin{equation}\label{1.10}
\| u \|_{L_{t}^{\frac{p + 1}{1 - s_{c}}} L_{x}^{p + 1}(\mathbb{R} \times \mathbb{R}^{d})},
\end{equation}
is bounded by a polynomial function of $\| u_{0} \|_{B_{1,1}^{\frac{d}{2} + s_{c}}}$.
\end{theorem}

The Besov space $B_{p, q}^{s}(\mathbb{R}^{d})$ is given by the norm
\begin{equation}\label{1.11}
\| u_{0} \|_{B_{p, q}^{s}(\mathbb{R}^{d})} = (\sum_{j} 2^{jps} \| P_{j} u_{0} \|_{L^{q}}^{p})^{1/p},
\end{equation}
when $1 \leq p < \infty$, with the usual modification when $p = \infty$. Here, $P_{j}$ is the usual Littlewood--Paley projection operator. The Sobolev embedding theorem implies that $B_{1,1}^{\frac{d}{2} + s_{c}}(\mathbb{R}^{d}) \subset \dot{H}^{s_{c}}(\mathbb{R}^{d})$. The $B_{1,1}^{\frac{d}{2} + s_{c}}(\mathbb{R}^{d})$ norm is invariant under the scaling symmetry $(\ref{1.9})$.

For the Schr{\"o}dinger equation in dimensions $d \geq 3$,
\begin{equation}\label{1.12}
i u_{t} + \Delta u = F, \qquad u(0,x) = u_{0}, \qquad u : I \times \mathbb{R}^{d} \rightarrow \mathbb{C},
\end{equation}
we have the Strichartz estimate
\begin{equation}\label{1.13}
\| u \|_{L_{t}^{2} L_{x}^{\frac{2d}{d - 2}} \cap L_{t}^{\infty} L_{x}^{2}(I \times \mathbb{R}^{d})} \lesssim \| u_{0} \|_{L^{2}} + \| F \|_{L_{t}^{1} L_{x}^{2} + L_{t}^{2} L_{x}^{\frac{2d}{d + 2}}(I \times \mathbb{R}^{d})}.
\end{equation}
See \cite{tao2006nonlinear} and the references therein for a detailed treatment of this topic.

In particular, when $F = 0$, $(\ref{1.13})$ implies a bound on $\| u \|_{L_{t}^{p} L_{x}^{q}}$, when $(p, q)$ is an admissible pair, i.e.
\begin{equation}\label{1.14}
\frac{2}{p} = d(\frac{1}{2} - \frac{1}{q}), \qquad p \geq 2.
\end{equation}
Then by the Sobolev embedding theorem, if $F = 0$,
\begin{equation}\label{1.15}
\| u \|_{L_{t}^{p} L_{x}^{r}} \lesssim \| u_{0} \|_{\dot{H}^{s}}, \qquad \text{for} \qquad \frac{1}{r} = \frac{1}{q} - \frac{s}{d}, \qquad (p, q) \qquad \text{is admissible}.
\end{equation}
The pair $(p, r)$ is then said to be $s$-admissible. Doing some algebra, $(\frac{p + 1}{1 - s_{c}}, p + 1)$ is $s_{c}$-admissible. Since $\frac{p + 1}{1 - s_{c}} < \infty$, a bound on $(\ref{1.10})$ on $\mathbb{R} \times \mathbb{R}^{d}$ implies scattering for $(\ref{1.1})$. Again, see \cite{tao2006nonlinear} for a detailed exposition of the method to prove this fact. It is useful to use the Strichartz space,
\begin{equation}\label{1.15.1}
\| u \|_{\dot{S}^{s_{c}}(I \times \mathbb{R}^{d})} = \| |\nabla|^{s_{c}} u \|_{L_{t}^{2} L_{x}^{\frac{2d}{d - 2}} \cap L_{t}^{\infty} L_{x}^{2}(I \times \mathbb{R}^{d})}.
\end{equation}

Studying nonlinear wave (\cite{dodson2018global}, \cite{dodson2021global2}) and Schr{\"o}dinger equations (\cite{dodson2021global}) with initial data in a critical Besov space has proved to be a very fruitful endeavor. Consider $(\ref{1.6})$, for example. Observe that the interaction Morawetz estimates in \cite{colliander2004global} and conservation of energy imply
\begin{equation}\label{1.16}
\| u \|_{L_{t}^{8} L_{x}^{4}(\mathbb{R} \times \mathbb{R}^{3})}^{8} \lesssim (1 + \| u_{0} \|_{\dot{H}^{1/2}}^{3}) \| u_{0} \|_{\dot{H}^{1}}^{3} \| u_{0} \|_{L^{2}}^{3}.
\end{equation}
This provides good bounds on the left hand side of $(\ref{1.16})$ for $u_{0} = e^{-\frac{|x|^{2}}{2}}$, or for any rescaled version under $(\ref{1.9})$, i.e., $u_{0} = \lambda e^{-\frac{\lambda^{2} |x|^{2}}{2}}$. Plugging
\begin{equation}\label{1.17}
u_{0} = c_{1} e^{-\frac{|x|^{2}}{2}} + c_{2} \lambda e^{-\lambda^{2} \frac{|x|^{2}}{2}},
\end{equation}
into $(\ref{1.16})$ gives a bound on the right hand side that is a function of $\lambda$, $c_{1}$, and $c_{2}$. However, Theorem $\ref{t1.1}$ implies that the $L_{t}^{8} L_{x}^{4}$ bound depends only on $c_{1}$ and $c_{2}$. Moreover, for $(\ref{1.1})$ with $1 < p < 3$, such bounds are polynomially dependent on $c_{1}$ and $c_{2}$. We could also replace $(\ref{1.17})$ by a convergent series of $c_{j}$'s, with arbitrary $\lambda_{j}$'s.

\section{Scattering for the cubic NLS in three dimensions}
We begin by proving scattering for the cubic equation $(\ref{1.9})$ with $u_{0} \in B_{1,1}^{2}(\mathbb{R}^{3})$, before moving on to the general problem. In this section, we do not prove any uniform bounds on the scattering size as a function of the $B_{1,1}^{2}$ norm of the initial data.
\begin{theorem}\label{t2.1}
The initial value problem
\begin{equation}\label{2.1}
i u_{t} + \Delta u = |u|^{2} u, \qquad u(0,x) = u_{0},
\end{equation}
with radially symmetric initial data $u_{0} \in B_{1,1}^{2}(\mathbb{R}^{3})$ has a global solution that scatters.
\end{theorem}
\begin{proof}
By time reversal symmetry, it suffices to prove scattering on $[0, \infty)$. In \cite{dodson2021global}, we proved that the cubic nonlinear Schr{\"o}dinger equation is globally well-posed for initial data $u_{0} \in W^{\frac{7}{6}, \frac{11}{7}}(\mathbb{R}^{3})$. By the Sobolev embedding theorem, $B_{1,1}^{2}(\mathbb{R}^{3}) \subset W^{\frac{7}{6}, \frac{11}{7}}(\mathbb{R}^{3})$, so global well-posedness follows.

Furthermore, after rescaling the initial data, suppose that the global solution has the form
\begin{equation}\label{2.2}
\| u \|_{L_{t,x}^{5}([0, 1] \times \mathbb{R}^{3})} \leq \delta, \qquad \text{which implies} \qquad \| u \|_{\dot{S}^{1/2}([0, 1] \times \mathbb{R}^{3})} < \infty.
\end{equation}
Then for $1 \leq t < \infty$, decompose
\begin{equation}\label{2.3}
u(t) = w(t) + v(t), \qquad \text{where} \qquad w(t) = e^{it \Delta} u_{0}^{(1)},
\end{equation}
and $u_{0} = u_{0}^{(1)} + u_{0}^{(2)}$ is some decomposition of $u_{0}$ that will be specified later.

Let $\mathcal E(t)$ denote the conformal energy of $v$,
\begin{equation}\label{2.4}
\mathcal E(t) = \| (x + 2it \nabla) v \|_{L^{2}}^{2} + 2t^{2} \| v \|_{L^{4}}^{4} = \| xv \|_{L^{2}}^{2} + 2 \langle xv, 2it \nabla v \rangle_{L^{2}} + 8 t^{2}E(t),
\end{equation}
where $E(t)$ is the energy in $(\ref{1.8})$,
\begin{equation}\label{2.5}
E(t) = \frac{1}{2} \| \nabla v \|_{L^{2}}^{2} + \frac{1}{4} \| v \|_{L^{4}}^{4}.
\end{equation}
When $w = 0$,
\begin{equation}\label{2.6}
\frac{d}{dt} \mathcal E(t) = -2t \| v \|_{L^{4}}^{4},
\end{equation}
which implies $\| v \|_{L_{t,x}^{4}([1, \infty) \times \mathbb{R}^{3})} < \infty$. Interpolating this bound with the bound $\| v \|_{L_{t}^{\infty} L_{x}^{4}} < \infty$, which a uniform bound on $\mathcal E(t)$ implies for any $1 \leq t \leq \infty$, gives
\begin{equation}\label{2.7}
\| u \|_{L_{t}^{8} L_{x}^{4}([1, \infty) \times \mathbb{R}^{3})} = \| v \|_{L_{t}^{8} L_{x}^{4}([1, \infty) \times \mathbb{R}^{3})} < \infty.
\end{equation}

For a general $u_{0}^{(1)} \in \dot{H}^{1/2}(\mathbb{R}^{3})$, Strichartz estimates imply that
\begin{equation}\label{2.8}
\| w \|_{L_{t}^{8} L_{x}^{4}(\mathbb{R} \times \mathbb{R}^{3})} \lesssim \| u_{0} \|_{\dot{H}^{1/2}(\mathbb{R}^{3})},
\end{equation}
so to prove scattering, it suffices to prove
\begin{equation}\label{2.9}
\int_{1}^{\infty} \frac{1}{t^{4}} \mathcal E(t)^{2} dt < \infty.
\end{equation}
Indeed,
\begin{equation}\label{2.9.1}
\int_{1}^{\infty} \| v \|_{L^{4}}^{8} dt \lesssim \int_{1}^{\infty} \frac{1}{t^{4}} \mathcal E(t)^{2} dt.
\end{equation}

By Duhamel's principle,
\begin{equation}\label{2.10}
v(1) = -i \int_{0}^{1} e^{i(1 - \tau) \Delta} |u|^{2} u d\tau + e^{i 1 \Delta} u_{0}^{(2)}.
\end{equation}
By direct computation,
\begin{equation}\label{2.11}
(x + 2i \nabla) \int_{0}^{1} e^{(1 - \tau) \Delta} |u|^{2} u d\tau = \int_{0}^{1} (x + 2i(1 - \tau) \nabla) e^{i(1 - \tau) \Delta} |u|^{2} u d\tau + \int_{0}^{1} 2i \tau \nabla e^{i(1 - \tau) \Delta} |u|^{2} u d\tau.
\end{equation}

By examining the kernel, 
\begin{equation}\label{2.12}
e^{it \Delta} f = \frac{C}{t^{3/2}} \int e^{-i \frac{|x - y|^{2}}{4t}} f(y) dy, \qquad (x + 2it \nabla) e^{it \Delta} f = e^{it \Delta} xf,
\end{equation}
so using the radial Sobolev embedding theorem and the computations in \cite{dodson2021global},
\begin{equation}\label{2.13}
\| \int_{0}^{1} (x + 2i(1 - \tau) \nabla) e^{i(1 - \tau) \Delta} |u|^{2} u d\tau \|_{L^{2}} \lesssim \| x |u|^{2} u \|_{L_{t}^{1} L_{x}^{2}} \lesssim \| xu \|_{L_{t,x}^{\infty}} \| u \|_{L_{t}^{8} L_{x}^{4}}^{2} \lesssim \| u \|_{L_{t}^{\infty} B_{1, 2}^{1/2}} \| u \|_{\dot{S}^{1/2}}^{2} \lesssim \| u_{0} \|_{B_{1,1}^{2}}^{3}.
\end{equation}
Next, recall from \cite{dodson2021global} that for any $0 < t < 1$,
\begin{equation}\label{2.14}
u = u_{1} + u_{2}, \qquad \text{where} \qquad \| \nabla u_{1} \|_{L^{2}} \lesssim t^{-1/4}, \qquad \| \nabla u_{2} \|_{L^{6}} \lesssim t^{-3/4},
\end{equation}
with constant independent of $t$. Therefore, by Strichartz estimates,
\begin{equation}\label{2.15}
\| \int_{0}^{1} 2i \tau \nabla e^{i(1 - \tau) \Delta} |u|^{2} u d\tau \|_{L^{2}} \lesssim \| \tau \nabla u_{1} \|_{L_{\tau}^{\infty} L_{x}^{2}} \| u \|_{L_{t}^{4} L_{x}^{6}}^{2} + \| \tau \nabla u_{2} \|_{L_{\tau}^{\infty} L_{x}^{6}} \| u \|_{L_{t}^{4} L_{x}^{6}}^{2} \lesssim_{\| u_{0} \|_{B_{1,1}^{2}}} 1.
\end{equation}

Now decompose the initial data. Let $\chi \in C_{0}^{\infty}(\mathbb{R}^{3})$, $\chi(x) = 1$ on $|x| \leq 1$, $\chi(x)$ is supported on $|x| \leq 2$, and let $R(\epsilon, u_{0}) < \infty$ be a constant sufficiently large so that
\begin{equation}\label{2.16}
\sum_{j} 2^{j/2} \| (1 - \chi(\frac{x}{R})) P_{j} u_{0} \|_{L^{2}} \leq \epsilon, \qquad \text{and} \qquad \sum_{j} 2^{2j} \| (1 - \chi(\frac{x}{R})) P_{j} u_{0} \|_{L^{1}} \leq \epsilon.
\end{equation}
By H{\"o}lder's inequality and the Sobolev embedding theorem,
\begin{equation}\label{2.17}
\| \nabla ((1 - \chi(\frac{x}{R})) P_{j} u_{0}) \|_{L^{2}} \lesssim \| P_{j} u_{0} \|_{\dot{H}^{1}}, \qquad \text{and} \qquad \| \nabla^{2} ((1 - \chi(\frac{x}{R})) P_{j} u_{0}) \|_{L^{1}} \lesssim 2^{2j} \| P_{j} u_{0} \|_{L^{1}}.
\end{equation}
Then,
\begin{equation}\label{2.18}
\sum_{j} \| (1 - \chi(\frac{x}{R})) P_{j} u_{0} \|_{L^{2}}^{1/2} \| (1 - \chi(\frac{x}{R})) P_{j} u_{0} \|_{\dot{H}^{1}}^{1/2} \lesssim \epsilon^{1/2} \| u_{0} \|_{B_{1,1}^{2}}^{1/2}.
\end{equation}
Therefore, by the radial Sobolev embedding theorem, if $\epsilon \leq \| u_{0} \|_{B_{1,1}^{2}}^{-2}$,
\begin{equation}\label{2.19}
\| |x| e^{it \Delta} (1 - \chi(\frac{x}{R})) u_{0} \|_{L^{\infty}} \lesssim \epsilon^{1/4}.
\end{equation}
Also, by H{\"o}lder's inequality,
\begin{equation}\label{2.20}
\| x \chi(\frac{x}{R}) u_{0} \|_{L^{2}} \lesssim R^{3/2} \| u_{0} \|_{L^{3}},
\end{equation}
so $(\ref{2.12})$ implies
\begin{equation}\label{2.21}
\| (x + 2i \nabla) v(1) \|_{L^{2}} \lesssim R^{3/2} \| u_{0} \|_{B_{1,1}^{2}}.
\end{equation}
The computations in \cite{dodson2021global} also imply
\begin{equation}\label{2.22}
\| v(1) \|_{L^{4}}^{4} \lesssim 1,
\end{equation}
and therefore,
\begin{equation}\label{2.23}
\mathcal E(1) \lesssim_{\| u_{0} \|_{B_{1,1}^{2}}, R} 1.
\end{equation}

To obtain the bound $(\ref{2.9})$, observe that $v$ solves
\begin{equation}\label{2.24}
i v_{t} + \Delta v = |u|^{2} u, \qquad v(1,x) = (\ref{2.10}),
\end{equation}
and $w$ solves
\begin{equation}\label{2.25}
i w_{t} + \Delta w = 0, \qquad w(1,x) = e^{i 1 \Delta} u_{0}^{(1)},
\end{equation}
on $[1, \infty)$.

Rearranging $(\ref{2.24})$,
\begin{equation}\label{2.26}
-\Delta v + |v|^{2} v = i v_{t} - F, \qquad F = 2 |v|^{2} w + v^{2} \bar{w} + 2 |w|^{2} v + w^{2} \bar{v} + |w|^{2} w = F_{1} + F_{2} + F_{3}.
\end{equation}
Integrating by parts,
\begin{equation}\label{2.27}
\aligned
\frac{d}{dt} \mathcal E(t) = 16t E(v) + 8t^{2} \langle v_{t}, -\Delta v + |v|^{2} v \rangle + 4 \langle xv, i \nabla v \rangle + 4t \langle x v_{t}, i \nabla v \rangle + 4t \langle xv, i \nabla v_{t} \rangle \\
+ 2 \langle ix \Delta v, xv \rangle - 2 \langle ixF, xv \rangle = -2t \| v \|_{L^{4}}^{4} + 8t^{2} \langle v_{t}, F \rangle - 4t \langle xF, \nabla v \rangle + 4t \langle xv, \nabla F \rangle - 2 \langle ixF, xv \rangle.
\endaligned
\end{equation}

Integrating by parts and plugging in $(\ref{2.26})$, with $F_{3} = |w|^{2} w$,
\begin{equation}\label{2.28}
\aligned
8t^{2} \langle v_{t}, F_{3} \rangle - 4t \langle xF_{3}, \nabla v \rangle + 4t \langle xv, \nabla F_{3} \rangle - 2 \langle ixF_{3}, xv \rangle \\
= 2\langle (x + 2it \nabla) |w|^{2} w, i(x + 2it \nabla) v \rangle_{L^{2}} + O(t^{2} \langle |v|^{3} + |w|^{3}, |w|^{3} \rangle) \\
\lesssim \| (x + 2it \nabla) v \|_{L^{2}} \| xw \|_{L^{\infty}} \| w \|_{L^{4}}^{2} + \| (x + 2it \nabla) v \|_{L^{2}} \| t \nabla w \|_{L^{\infty}} \| w \|_{L^{4}}^{2} + t^{2} \| v \|_{L^{4}}^{3} \| w \|_{L^{4}} \| w \|_{L^{\infty}}^{2} + t^{2} \| w \|_{L^{6}}^{6} \\
\lesssim \mathcal E(t)^{1/2} \| w \|_{L^{4}}^{2} + t^{3/8} \mathcal E(t)^{3/4} \| w \|_{L^{\infty}}^{2} + t^{3/2} \| w \|_{L^{\infty}}^{2}.
\endaligned
\end{equation}
Also, integrating by parts and plugging in $(\ref{2.26})$ with $F_{2} = 2 |w|^{2} v + w^{2} \bar{v}$,
\begin{equation}\label{2.29}
\aligned
8t^{2} \langle v_{t}, F_{2} \rangle - 4t \langle xF_{2}, \nabla v \rangle + 4t \langle xv, \nabla F_{2} \rangle - 2 \langle ixF_{2}, xv \rangle \\
= 2\langle (x + 2it \nabla) F_{2}, i(x + 2it \nabla) v \rangle_{L^{2}} + O(t^{2} \langle |v|^{3} + |w|^{3}, |w|^{2} |v| \rangle) \\
\lesssim \| (x + 2it \nabla) v \|_{L^{2}} \| xw \|_{L^{\infty}} \| w \|_{L^{4}} \| v \|_{L^{4}} + \| (x + 2it \nabla) v \|_{L^{2}} \| t \nabla w \|_{L^{\infty}} \| w \|_{L^{4}} \| v \|_{L^{4}} \\ + \| (x + 2it \nabla) v \|_{L^{2}}^{2} \| w \|_{L^{\infty}}^{2} + t^{2} \| v \|_{L^{4}}^{4} \| w \|_{L^{\infty}}^{2} + t^{2} \| w \|_{L^{6}}^{6} \\
\lesssim t^{-1/2} \mathcal E(t)^{3/4} \| w \|_{L^{4}} + t^{3/2} \| w \|_{L^{\infty}}^{2} + \mathcal E(t) \| w \|_{L^{\infty}}^{2}.
\endaligned
\end{equation}

Finally, take
\begin{equation}\label{2.30}
8t^{2} \langle v_{t}, F_{1} \rangle - 4t \langle xF_{1}, \nabla v \rangle + 4t \langle xv, \nabla F_{1} \rangle - 2 \langle ixF_{1}, xv \rangle,
\end{equation}
with $F_{1} = 2 |v|^{2} w + v^{2} \bar{w}$. This term will be handled slightly differently from $(\ref{2.28})$ and $(\ref{2.29})$. By $(\ref{2.19})$,
\begin{equation}\label{2.31}
- 4t \langle xF_{1}, \nabla v \rangle - 2 \langle ixF_{1}, xv \rangle = -2 \langle ix F_{1}, (x + 2it \nabla) v \rangle \lesssim \| (x + 2it \nabla) v \|_{L^{2}} \| v \|_{L^{4}}^{2} \| x w \|_{L^{\infty}} \lesssim \frac{\epsilon^{1/4}}{t} \mathcal E(t).
\end{equation}
Next, integrating by parts,
\begin{equation}\label{2.32}
\aligned
4t \langle xv, \nabla F_{1} \rangle = -4t \langle x \nabla v, F_{1} \rangle - 12t \langle v, F_{1} \rangle = -4t \langle xw, \nabla(|v|^{2} v) \rangle -12t \langle v, F_{1} \rangle \\ \lesssim t \| w \|_{L^{4}} \| v \|_{L^{4}}^{3} + 4t \langle \nabla w, x |v|^{2} v \rangle \lesssim t^{-1/2} \| w \|_{L^{4}} \mathcal E(t)^{3/4} + 4t \langle \nabla w, x |v|^{2} v \rangle.
\endaligned
\end{equation}
Then by the product rule, integrating by parts, and $(\ref{2.17})$,
\begin{equation}\label{2.33}
\aligned
4t \langle \nabla w, x |v|^{2} v \rangle = 8t \langle \nabla w, |v|^{2} (x + 2it \nabla) v \rangle - 4t^{2} \langle \nabla w, v^{2} (x - 2it \nabla) \bar{v} \rangle - 8t \langle \nabla w, i \nabla(|v|^{2} v) \rangle \\
\lesssim t \| \nabla w \|_{L^{\infty}} \| (x + 2it \nabla) v \|_{L^{2}} \| v \|_{L^{4}}^{2} - 8t^{2} \langle i \Delta w, |v|^{2} v \rangle \\ \lesssim \| \nabla w \|_{L^{\infty}} \mathcal E(t) + t \| \Delta w \|_{L^{\infty}} \| v \|_{L^{2}} \mathcal E(t)^{1/2} \lesssim \| \nabla w \|_{L^{\infty}} \mathcal E(t) + t^{-1/2} \| v \|_{L^{2}} \mathcal E(t)^{1/2}.
\endaligned
\end{equation}
Meanwhile, integrating by parts in $t$,
\begin{equation}\label{2.34}
\int_{1}^{T} 8t^{2} \langle v_{t}, F_{1} \rangle dt = 8t^{2} \langle |v|^{3}, |w| \rangle|_{1}^{T} - \int_{1}^{T} 8t^{2} \langle |v|^{2} v, w_{t} \rangle - \int_{1}^{T} 16t \langle |v|^{2} v, w \rangle dt.
\end{equation}
First observe that
\begin{equation}\label{2.35}
8t^{2} \langle |v|^{3}, |w| \rangle|_{1}^{T} \lesssim t^{1/2} \| w \|_{L^{4}} \mathcal E(t)^{3/4}|_{1}^{T}.
\end{equation}
Also compute
\begin{equation}\label{2.36}
8t^{2}  \langle |v|^{2} v, w_{t} \rangle \lesssim t \| \Delta w \|_{L^{\infty}} \| v \|_{L^{2}} \mathcal E(t)^{1/2} \lesssim t^{-1/2} \| v \|_{L^{2}} \mathcal E(t)^{1/2}, \qquad \text{and} \qquad 16t \langle |v|^{2} v, w \rangle \lesssim t^{-1/2} \mathcal E(t)^{3/4} \| w \|_{L^{4}}.
\end{equation}

Therefore,
\begin{equation}\label{2.37}
\aligned
\mathcal E(t) \lesssim \int_{1}^{t} [\mathcal E(s)^{1/2} \| w \|_{L^{4}}^{2} + s^{3/8} \mathcal E(s)^{3/4} \| w \|_{L^{\infty}}^{2} + s^{3/2} \| w \|_{L^{\infty}}^{2} + s^{1/4} \mathcal E(s)^{1/2} \| w \|_{L^{\infty}} \\ + \mathcal E(s) \| w \|_{L^{\infty}}^{2} + \frac{\epsilon}{s} \mathcal E(s) + \| \nabla w \|_{L^{\infty}} \mathcal E(s) + s^{-1/2} \mathcal E(s)^{1/2} \| v \|_{L^{2}} + s^{-1/2} \mathcal E(s)^{3/4} \| w \|_{L^{4}}] ds + t^{1/2} \| w \|_{L^{4}} \mathcal E(t)^{3/4} + R.
\endaligned
\end{equation}
By Fubini's theorem and H{\"o}lder's inequality,
\begin{equation}\label{2.38}
\aligned
 \int_{1}^{\infty} \frac{1}{t^{4}} (\int_{1}^{t} \mathcal E(s)^{1/2} \| w \|_{L^{4}}^{2} ds)^{2} dt \lesssim \int_{1}^{\infty} \frac{1}{t^{3}} (\int_{1}^{t} \mathcal E(s) \| w \|_{L^{4}}^{4} ds) dt = \int_{1}^{\infty} \mathcal E(s) \| w(s) \|_{L^{4}}^{4} \int_{s}^{\infty} \frac{1}{t^{3}} dt ds \\
 \lesssim \int_{1}^{\infty} \frac{1}{s^{2}} \mathcal E(s) \| w \|_{L^{4}}^{4} ds \lesssim (\int_{1}^{\infty} \frac{1}{s^{4}} \mathcal E(s)^{2} ds)^{1/2} (\int_{1}^{\infty} \| w \|_{L^{4}}^{8} ds)^{1/2}.
 \endaligned
\end{equation}
Next, interpolating $(\ref{2.16})$ and $(\ref{2.17})$,
\begin{equation}\label{2.39}
\aligned
\| \nabla e^{it \Delta} (1 - \chi(\frac{x}{R})) P_{j} u_{0} \|_{L^{\infty}} \\ \lesssim \inf \{ t^{-3/2} 2^{-j} \cdot 2^{j} \| (1 - \chi(\frac{x}{R})) (P_{j} u_{0}) \|_{L^{1}}^{1/2} \| \nabla^{2} (1 - \chi(\frac{x}{R})) (P_{j} u_{0}) \|_{L^{1}}^{1/2}, 2^{j} \| \nabla^{2} (1 - \chi(\frac{x}{R})) (P_{j} u_{0}) \|_{L^{1}} \},
\endaligned
\end{equation}
which by $(\ref{2.18})$ implies that for $\epsilon \leq \| u_{0} \|_{B_{1,1}^{2}}^{-8}$,
\begin{equation}\label{2.40}
\int_{0}^{\infty} \| \nabla w \|_{L^{\infty}} dt \lesssim \epsilon^{1/4} \| u_{0} \|_{B_{1,1}^{2}}^{3/4} \lesssim \epsilon^{5/32}.
\end{equation}
Similar computations also show that
\begin{equation}\label{2.41}
\| w \|_{L_{t}^{2} L_{x}^{\infty}} \lesssim \epsilon^{3/8}, \qquad \text{and} \qquad \| w \|_{L^{\infty}} \lesssim \frac{\epsilon^{3/8}}{s^{1/2}}.
\end{equation}

Therefore,
\begin{equation}\label{2.42}
\aligned
\int_{1}^{\infty} \frac{1}{t^{4}} \mathcal E(t)^{2} dt \lesssim \int_{1}^{\infty} \frac{R^{2}}{t^{4}} dt + \int_{1}^{\infty} \frac{1}{t^{3}} \mathcal E(t)^{3/2} \| w \|_{L^{4}}^{2} dt + \int_{1}^{\infty} \frac{1}{t^{4}} (\int_{1}^{t} \mathcal E(s)^{1/2} \| w \|_{L^{4}}^{2} ds)^{2} dt \\
+ \int_{1}^{\infty} \frac{1}{t^{4}} (\int_{1}^{t} s^{3/8} \mathcal E(s)^{3/4} \| w \|_{L^{\infty}}^{2} + s^{3/2} \| w \|_{L^{\infty}}^{2} + s^{1/4} \mathcal E(s)^{1/2} \| w \|_{L^{\infty}} ds)^{2} dt \\
+ \int_{1}^{\infty} \frac{1}{t^{4}}(\int_{1}^{t} \mathcal E(s) \| w \|_{L^{\infty}}^{2} + \frac{\epsilon}{s} \mathcal E(s) + s^{-1/2} \mathcal E(s)^{1/2} \| v \|_{L^{2}} + s^{-1/2} \mathcal E(s)^{3/4} \| w \|_{L^{4}} ds)^{2} dt \\
\lesssim R^{2} + (\int_{1}^{\infty} \frac{1}{t^{4}} \mathcal E(t)^{2} dt)^{3/4}(\int_{1}^{\infty} \| w \|_{L^{4}}^{8} dt)^{1/4} + (\int_{1}^{\infty} \frac{1}{s^{4}} \mathcal E(s)^{2} ds)^{1/2}(\int_{1}^{\infty} \| w \|_{L^{4}}^{8} ds)^{1/2} \\
+ (\int_{1}^{\infty} \frac{1}{s^{4}} \mathcal E(s)^{2} ds)^{3/4} (\int_{1}^{\infty} s^{7} \| w \|_{L^{\infty}}^{16} ds)^{1/4} + (\int_{1}^{\infty} s \| w \|_{L^{\infty}}^{4} ds) + (\int_{1}^{\infty} \frac{1}{s^{4}} \mathcal E(s)^{2} ds)^{1/2} (\int_{1}^{\infty} \| w \|_{L^{\infty}}^{4} ds)^{1/2} \\
+ \epsilon^{5/16} (\int_{1}^{\infty} \frac{1}{s^{4}} \mathcal E(s)^{2} ds) + (\int_{1}^{\infty} \frac{1}{s^{4}} \mathcal E(s)^{2} ds)^{1/2} (\int_{1}^{\infty} \frac{1}{s^{2}} \| v(s) \|_{L^{2}}^{4} ds)^{1/2}.
\endaligned
\end{equation}
Therefore,
\begin{equation}\label{2.43}
\int_{1}^{\infty} \frac{1}{t^{4}} \mathcal E(t)^{2} dt \lesssim R^{2} + \int_{1}^{\infty} \| w \|_{L^{4}}^{8} dt + \int_{1}^{\infty} \| w \|_{L^{\infty}}^{2} dt + (\int_{1}^{\infty} \frac{1}{t^{4}} \mathcal E(t)^{2} dt)^{1/2} (\int_{1}^{\infty} \frac{1}{t^{2}} \| v(t) \|_{L^{2}}^{2} dt)^{1/2}.
\end{equation}
Now since $v$ solves $(\ref{2.24})$,
\begin{equation}\label{2.44}
\frac{d}{dt} \| v \|_{L^{2}}^{2} \lesssim \| w \|_{L^{\infty}} \| v \|_{L^{4}}^{2} \| v \|_{L^{2}}^{2} + \| w \|_{L^{\infty}} \| w \|_{L^{4}}^{2} \| v \|_{L^{2}}, \qquad \| v(1) \|_{L^{2}}^{2} \lesssim R.
\end{equation}
Therefore, by H{\"o}lder's inequality,
\begin{equation}\label{2.45}
\aligned
\| v(t) \|_{L^{2}}^{4} \lesssim R^{2} + (\int_{1}^{t} \| w \|_{L^{\infty}} \| v \|_{L^{4}}^{2} \| v \|_{L^{2}} + \| w \|_{L^{\infty}} \| w \|_{L^{4}}^{2} \| v \|_{L^{2}} dt)^{2} \\ \lesssim R^{2} + \| w \|_{L_{t}^{2} L_{x}^{\infty}}^{2} (\int_{1}^{t} \| v \|_{L^{4}}^{4} \| v \|_{L^{2}}^{2} + \| w \|_{L^{4}}^{4} \| v \|_{L^{2}}^{2}).
\endaligned
\end{equation}
Therefore, by Fubini's theorem,
\begin{equation}\label{2.46}
\aligned
\int_{1}^{\infty} \frac{1}{t^{2}} \| v(t) \|_{L^{2}}^{4} dt \lesssim R^{2} + \epsilon^{2} (\int_{1}^{\infty} \| v \|_{L^{4}}^{8} dt)^{1/2} (\int_{1}^{\infty} \frac{1}{t^{2}} \| v(t) \|_{L^{2}}^{4} dt)^{1/2} + \epsilon^{2} \| w \|_{L_{t}^{8} L_{x}^{4}}^{4} (\int_{1}^{\infty} \frac{1}{t^{2}} \| v(t) \|_{L^{2}}^{4} dt)^{1/2} \\
\lesssim R^{2} + \epsilon^{2} (\int_{1}^{\infty} \frac{1}{t^{4}} \mathcal E(t)^{2} dt)^{1/2} (\int_{1}^{\infty} \frac{1}{t^{2}} \| v(t) \|_{L^{2}}^{4} dt)^{1/2} + \epsilon^{2} \| w \|_{L_{t}^{8} L_{x}^{4}}^{4} (\int_{1}^{\infty} \frac{1}{t^{2}} \| v(t) \|_{L^{2}}^{4} dt)^{1/2}.
\endaligned
\end{equation}
Therefore,
\begin{equation}\label{2.47}
\int_{1}^{\infty} \frac{1}{t^{2}} \| v(t) \|_{L^{2}}^{4} dt \lesssim R^{2} + \epsilon^{2} (\int_{1}^{\infty} \frac{1}{t^{4}} \mathcal E(t)^{2} dt) + \epsilon^{2} \int_{1}^{\infty} \| w \|_{L^{4}}^{8} dt.
\end{equation}
Plugging $(\ref{2.47})$ into $(\ref{2.43})$,
\begin{equation}\label{2.48}
\int_{1}^{\infty} \| v(t) \|_{L^{4}}^{8} dt \lesssim \int_{1}^{\infty} \frac{1}{t^{4}} \mathcal E(t)^{2} dt \lesssim R^{2} + \int_{1}^{\infty} \| w \|_{L^{4}}^{8} dt + \int_{1}^{\infty} \| w \|_{L^{\infty}}^{2} dt.
\end{equation}
Therefore, scattering follows.
\end{proof}

\section{Concentration compactness in the cubic case}
Recall from Strichartz estimates that for a solution to $(\ref{1.6})$,
\begin{equation}
\| u \|_{L_{t}^{8} L_{x}^{4}(\mathbb{R} \times \mathbb{R}^{3})} < \infty, \qquad \text{is equivalent to} \qquad \| u \|_{L_{t,x}^{5}(\mathbb{R} \times \mathbb{R}^{3})} < \infty,
\end{equation}
so Theorem $\ref{t2.1}$ implies that for $u_{0} \in B_{1,1}^{2}$, $(\ref{1.1})$ has a global solution satisfying $\| u \|_{L_{t,x}^{5}(\mathbb{R} \times \mathbb{R}^{3})} < \infty$. However, since $R$ depends on $\epsilon > 0$ and $u_{0}$, not just the norm $\| u_{0} \|_{B_{1,1}^{2}}$, $(\ref{2.48})$ does not directly give a uniform bound on
\begin{equation}\label{3.0}
\| u \|_{L_{t,x}^{5}(\mathbb{R} \times \mathbb{R}^{3})}, \qquad \text{when} \qquad \| u_{0} \|_{B_{1,1}^{2}} \leq A < \infty.
\end{equation}
Such a bound follows from a concentration compactness argument, as in \cite{dodson2018global} for the nonlinear wave equation. 

Following by now standard concentration compactness techniques, see for example \cite{keraani2001defect},
\begin{lemma}\label{l3.1}
Let $u_{n}$ be a bounded sequence in $\dot{H}^{1/2}$,
\begin{equation}\label{3.1}
\sup_{n} \| u_{n} \|_{\dot{H}^{1/2}(\mathbb{R}^{3})} \leq A < \infty,
\end{equation}
that is radially symmetric. After passing to a subsequence, assume that
\begin{equation}\label{3.2}
\lim_{n \rightarrow \infty} \| u_{n} \|_{\dot{H}^{1/2}(\mathbb{R}^{3})} = A.
\end{equation}
Then passing to a further subsequence, for any $1 \leq J < \infty$, there exist $\phi^{1}$, ..., $\phi^{J} \in \dot{H}^{1/2}$ such that
\begin{equation}\label{3.3}
u_{n} = \sum_{j = 1}^{J} e^{i t_{n}^{j} (\lambda_{n}^{j})^{2} \Delta} \frac{1}{\lambda_{n}^{j}} \phi^{j}(\frac{x}{\lambda_{n}^{j}}) + w_{n}^{J},
\end{equation}
where
\begin{equation}\label{3.4}
\sum_{j = 1}^{J} \| \phi^{j} \|_{\dot{H}^{1/2}}^{2} + \lim_{n \rightarrow \infty} \| w_{n}^{J} \|_{\dot{H}^{1/2}}^{2} = A^{2},
\end{equation}
\begin{equation}\label{3.5}
\lim_{J \rightarrow \infty} \limsup_{n \rightarrow \infty} \| e^{it \Delta} w_{n}^{J} \|_{L_{t,x}^{5}(\mathbb{R} \times \mathbb{R}^{3})} = 0,
\end{equation}
and for $j \neq k$,
\begin{equation}\label{3.6}
\lim_{n \rightarrow \infty} |\ln(\frac{\lambda_{n}^{j}}{\lambda_{n}^{k}})| + |t_{n}^{j} - t_{n}^{k}| = \infty.
\end{equation}
\end{lemma}

Now let $u_{n}$ be a sequence in $B_{1,1}^{2}(\mathbb{R}^{3})$ with the uniform bound
\begin{equation}\label{3.7}
\| u_{n} \|_{B_{1,1}^{2}} \leq A.
\end{equation}
Then by the Sobolev embedding theorem,
\begin{equation}\label{3.8}
\| u_{n} \|_{\dot{H}^{1/2}(\mathbb{R}^{3})} \lesssim A,
\end{equation}
so apply Lemma $\ref{l3.1}$, and observe that for any $J$,
\begin{equation}\label{3.9}
u_{n} = \sum_{j = 1}^{J} e^{i t_{n}^{j} (\lambda_{n}^{j})^{2} \Delta} \frac{1}{\lambda_{n}^{j}} \phi^{j}(\frac{x}{\lambda_{n}^{j}}) + w_{n}^{J}.
\end{equation}

Next, observe that Lemma $\ref{l3.1}$ implies that for any fixed $j$,
\begin{equation}\label{3.10}
e^{-i t_{n}^{j} \Delta} (\lambda_{n}^{j} u_{n}( \lambda_{n}^{j} \cdot)) \rightharpoonup \phi^{j}, \qquad \text{weakly in} \qquad \dot{H}^{1/2}(\mathbb{R}^{3}).
\end{equation}
Using dispersive estimates, for any $t \in \mathbb{R}$, since $B_{1,1}^{2}$ is invariant under the scaling symmetry $(\ref{1.9})$,
\begin{equation}\label{3.11}
\| e^{it \Delta} e^{-i t_{n}^{j} \Delta} (\lambda_{n}^{j} u_{n}(\lambda_{n}^{j} \cdot)) \|_{L^{\infty}} \lesssim \frac{1}{|t - t_{n}^{j}|^{1/2}} \| u_{n} \|_{B_{1,1}^{2}},
\end{equation}
in particular, if $t_{n}^{j} \rightarrow \pm \infty$ along a subsequence, interpolating $(\ref{3.11})$ and the Sobolev embedding theorem $\dot{H}^{1/2} \hookrightarrow L^{3}$,
\begin{equation}\label{3.12}
\| e^{it \Delta} e^{-i t_{n}^{j} \Delta} (\lambda_{n}^{j} u_{n}(\lambda_{n}^{j} \cdot)) \|_{L_{t,x}^{5}([-T, T] \times \mathbb{R}^{3})} = 0,
\end{equation}
for any fixed $0 < T < \infty$. Since $u_{n} \rightharpoonup \phi$ weakly in $\dot{H}^{1/2}$ implies
\begin{equation}\label{3.13}
e^{it \Delta} u_{n} \rightharpoonup e^{it \Delta} \phi, \qquad \text{weakly in} \qquad L_{t,x}^{5},
\end{equation}
$(\ref{3.12})$ implies that $\phi^{j} = 0$ if $t_{n}^{j} \rightarrow \pm \infty$ along a subsequence.
\begin{remark}
The fact that weak convergence implies $(\ref{3.13})$ follows from Strichartz estimates and approximating a function in $L_{t,x}^{5/4}$ with a smooth, compactly supported function and a small remainder.
\end{remark}

Therefore, the $t_{n}^{j}$'s must be uniformly bounded for any $j$, and after passing to a subsequence, $t_{n}^{j} \rightarrow t^{j} \in \mathbb{R}$ for any $j$. Since
\begin{equation}\label{3.14}
e^{i t_{n}^{j} (\lambda_{n}^{j})^{2} \Delta} \frac{1}{\lambda_{n}^{j}} \phi^{j}(\frac{x}{\lambda_{n}^{j}}) = \frac{1}{\lambda_{n}^{j}} (e^{i t_{n}^{j} \Delta} \phi^{j})(\frac{x}{\lambda_{n}^{j}}),
\end{equation}
replacing $\phi^{j}$ with $e^{it^{j} \Delta} \phi^{j}$ and absorbing the remainder into $w_{n}^{J}$, it is possible to set $t_{n}^{j} \equiv 0$ for all $j$ in $(\ref{3.9})$. Therefore,
\begin{equation}\label{3.15}
u_{n} = \sum_{j = 1}^{J} \frac{1}{\lambda_{n}^{j}} \phi^{j}(\frac{x}{\lambda_{n}^{j}}) + w_{n}^{J}.
\end{equation}
By Theorem $\ref{t2.1}$, for any $j$, let $u^{j}$ be the solution to $(\ref{2.1})$ with initial data $\phi^{j}$. Then for any $j$,
\begin{equation}\label{3.16}
\| u^{j} \|_{L_{t,x}^{5}(\mathbb{R} \times \mathbb{R}^{3})} < \infty.
\end{equation}

Furthermore, $(\ref{3.4})$, $(\ref{3.5})$, $(\ref{3.6})$, and small data arguments imply that if $u^{(n)}(t,x)$ is the solution to $(\ref{2.1})$ with initial data $u_{n}(x)$,
\begin{equation}\label{3.17}
\lim_{n \rightarrow \infty} \| u^{(n)} \|_{L_{t,x}^{5}(\mathbb{R} \times \mathbb{R}^{3})}^{5} \leq \sum_{j = 1}^{\infty} \| u^{j} \|_{L_{t,x}^{5}(\mathbb{R} \times \mathbb{R}^{3})}^{5} < \infty.
\end{equation}
For all but finitely many $j$'s, say all but $j_{0}$, $\| u^{j} \|_{L_{t}^{\infty} \dot{H}^{1/2}} \leq \epsilon$, so by small data arguments and $(\ref{3.4})$,
\begin{equation}
\sum_{j \geq j_{0}} \| u^{j} \|_{L_{t,x}^{5}(\mathbb{R} \times \mathbb{R}^{3})}^{2} \lesssim A.
\end{equation}
Therefore, there exists a function $f : [0, \infty) \rightarrow [0, \infty)$ such that if $\| u_{0} \|_{B_{1,1}^{2}} \leq A$ is radial, then $(\ref{2.1})$ has a global solution that satisfies the bound
\begin{equation}\label{3.18}
\| u \|_{L_{t,x}^{5}(\mathbb{R} \times \mathbb{R}^{3})} \leq f(A) < \infty.
\end{equation}

\medskip

Observe that $(\ref{3.18})$ gives no explicit bound on the scattering size. In general, the bounds obtained from a concentration compactness argument are likely far from optimal. For example, in \cite{MR2154347},
\begin{equation}\label{3.19}
\| u \|_{L_{t,x}^{\frac{2(d + 2)}{d - 2}}(\mathbb{R} \times \mathbb{R}^{d})} \leq C \exp(C E^{C}),
\end{equation}
where $C(d)$ is a large constant, $E$ is the energy $(\ref{1.8})$, and $u$ is a solution to the energy-critical problem ($s_{c} = 1$) with radially symmetric initial data.

\section{A local result for $(\ref{1.1})$ when $1 < p < 3$}

In the second part of the paper, we will prove explicit bounds on the scattering size of a solution to $(\ref{1.1})$ with radially symmetric initial data in $B_{1,1}^{\frac{d}{2} + s_{c}}$, when $0 < s_{c} < 1$ and $1 < p < 3$. Note that the restrictions on $s_{c}$ and $p$ require $d \geq 3$.

As in the cubic case, the first step is to rescale and obtain good bounds on the interval $[0, 1]$. The space $L_{t,x}^{\frac{(d + 2)(p - 1)}{2}}(\mathbb{R} \times \mathbb{R}^{d})$ is also invariant under the rescaling $(\ref{1.9})$, so rescale the initial data so that
\begin{equation}\label{4.1}
\| u \|_{L_{t,x}^{\frac{(d + 2)(p - 1)}{2}}([0, 1] \times \mathbb{R}^{d})} \leq \delta,
\end{equation}
for some $\delta \ll 1$.

\begin{lemma}\label{l4.1}
If $u$ is a solution to $(\ref{1.1})$ on $[0, 1]$ with initial data $u_{0} \in B_{1,1}^{\frac{d}{2} + s_{c}}$, and $u$ satisfies $(\ref{4.1})$, then for any $j \in \mathbb{Z}_{< 0}$,
\begin{equation}\label{4.2}
\| \nabla u \|_{L_{t}^{2} L_{x}^{\frac{2d}{d - 2}}([2^{j}, 2^{j + 1}] \times \mathbb{R}^{d})} \lesssim 2^{j \frac{s_{c} - 1}{2}} \| u_{0} \|_{B_{1,1}^{\frac{d}{2} + s_{c}}(\mathbb{R}^{d})},
\end{equation}
\end{lemma}
\begin{proof}
The local solution may be obtained by showing that the operator
\begin{equation}\label{4.3}
\Phi(u(t)) = e^{it \Delta} u_{0} - i \int_{0}^{t} e^{i(t - \tau) \Delta} |u(\tau)|^{p - 1} u(\tau) d\tau,
\end{equation}
has a unique fixed point in $\dot{S}^{s_{c}}([0, 1] \times \mathbb{R}^{d})$.

Interpolating the Sobolev embedding theorem,
\begin{equation}\label{4.4}
\| P_{k} e^{it \Delta} u_{0} \|_{L^{\infty}} \lesssim 2^{k \frac{2}{p - 1}} 2^{k(d - \frac{2}{p - 1})} \| P_{k} u_{0} \|_{L^{1}},
\end{equation}
with the dispersive estimate,
\begin{equation}\label{4.5}
\| P_{k} e^{it \Delta} u_{0} \|_{L^{\infty}} \lesssim t^{-d/2} 2^{-k(d - \frac{2}{p - 1})} 2^{k(d - \frac{2}{p - 1})} \| P_{k} u_{0} \|_{L^{1}},
\end{equation}
where $P_{k}$ is the usual Littlewood--Paley projection operator for any $k \in \mathbb{Z}$,
\begin{equation}\label{4.6}
\| e^{it \Delta} u_{0} \|_{L^{\infty}} \lesssim t^{-\frac{1}{p - 1}} \| u_{0} \|_{B_{1,1}^{\frac{d}{2} + s_{c}}},
\end{equation}
and
\begin{equation}\label{4.7}
\| \nabla e^{it \Delta} u_{0} \|_{L^{\infty}} \lesssim t^{-\frac{1}{p - 1} - \frac{1}{2}} \| u_{0} \|_{B_{1,1}^{\frac{d}{2} + s_{c}}}.
\end{equation}

Interpolating $(\ref{4.5})$ with the Sobolev embedding theorem,
\begin{equation}\label{4.8}
\| \nabla P_{k} e^{it \Delta} u_{0} \|_{L^{2}} \lesssim 2^{k(1 - s_{c})} 2^{k(\frac{d}{2} + s_{c})} \| P_{k} u_{0} \|_{L^{1}},
\end{equation}
and
\begin{equation}\label{4.9}
\| \nabla P_{k} e^{it \Delta} u_{0} \|_{L^{\frac{2d}{d - 2}}} \lesssim 2^{-k s_{c}} \frac{1}{t} 2^{k(\frac{d}{2} + s_{c})} \| P_{k} u_{0} \|_{L^{1}}.
\end{equation}
Interpolating this bound with
\begin{equation}\label{4.10}
\| \nabla P_{k} e^{it \Delta} u_{0} \|_{L^{\frac{2d}{d - 2}}} \lesssim 2^{k(2 - s_{c})} 2^{k(\frac{d}{2} + s_{c})} \| P_{k} u_{0} \|_{L^{1}},
\end{equation}
we obtain
\begin{equation}\label{4.11}
\| \nabla e^{it \Delta} u_{0} \|_{L^{\frac{2d}{d - 2}}} \lesssim t^{-\frac{1}{2} - \frac{1 - s_{c}}{2}} \| u_{0} \|_{B_{1,1}^{\frac{d}{2} + s_{c}}}.
\end{equation}
Therefore, for any $j \in \mathbb{Z}_{< 0}$,
\begin{equation}\label{4.12}
\| \nabla e^{it \Delta} u_{0} \|_{L_{t}^{2} L_{x}^{\frac{2d}{d - 2}}([2^{j}, 2^{j + 1}] \times \mathbb{R}^{d})} \lesssim 2^{j \frac{s_{c} - 1}{2}} \| u_{0} \|_{B_{1,1}^{\frac{d}{2} + s_{c}}}.
\end{equation}

By Strichartz estimates, for any $t \in [2^{j}, 2^{j + 1}]$, let $j_{\delta}$ be the integer closest to $\log_{2}(\delta 2^{j})$. By Strichartz estimates and the chain rule, and by $(\ref{4.1})$,
\begin{equation}\label{4.13}
\| \nabla \int_{2^{j_{\delta}}}^{t} e^{i(t - \tau) \Delta} |u(\tau)|^{p - 1} u(\tau) d\tau \|_{L_{t}^{2} L_{x}^{\frac{2d}{d - 2}}([2^{j}, 2^{j + 1}] \times \mathbb{R}^{d})} \lesssim \delta^{2} \| \nabla u \|_{L_{t}^{2} L_{x}^{\frac{2d}{d - 2}}([2^{j_{\delta}}, 2^{j + 1}] \times \mathbb{R}^{d})}.
\end{equation}
Meanwhile, the dispersive estimate combined with the Sobolev embedding theorem $\| u \|_{L_{t}^{\infty} L_{x}^{\frac{d(p - 1)}{2}}([0, 1] \times \mathbb{R}^{d})} \lesssim \| u_{0} \|_{B_{1,1}^{\frac{d}{2} + s_{c}}}$,
\begin{equation}\label{4.14}
\| \int_{0}^{2^{j_{\delta}}} e^{i(t - \tau) \Delta} |u(\tau)|^{p - 1} u(\tau) d\tau \|_{L^{\frac{2d}{d - 2}}} \lesssim \| u \|_{L_{t}^{\infty} L_{x}^{\frac{d(p - 1)}{2}}([0, 1] \times \mathbb{R}^{d})}^{p - 1} \| \nabla u \|_{L_{t}^{1} L_{x}^{\frac{2d}{d - 2}}([0, 2^{j_{\delta}}] \times \mathbb{R}^{d})}.
\end{equation}
Therefore, for $\delta(\| u_{0} \|_{B_{1,1}^{\frac{d}{2} + s_{c}}}) > 0$ sufficiently small,
\begin{equation}\label{4.15}
\sup_{j < 0} 2^{j \frac{1 - s_{c}}{2}}  \| \nabla \Phi(u) \|_{L_{t}^{2} L_{x}^{\frac{2d}{d - 2}}([2^{j}, 2^{j + 1}] \times \mathbb{R}^{d})} \lesssim \| u_{0} \|_{B_{1,1}^{\frac{d}{2} + s_{c}}} + \epsilon \cdot \sup_{j < 0} 2^{j \frac{1 - s_{c}}{2}}  \| \nabla u \|_{L_{t}^{2} L_{x}^{\frac{2d}{d - 2}}([2^{j}, 2^{j + 1}] \times \mathbb{R}^{d})},
\end{equation}
for some $\epsilon > 0$. Thus, $(\ref{4.2})$ holds.
\end{proof}

Now suppose $(\ref{1.1})$ with initial data $u_{0}$ has a solution on the maximal interval $[0, T)$, where $1 < T \leq \infty$. Again decompose $u = v + w$, where $v$ and $w$ solve
\begin{equation}\label{4.16}
i v_{t} + \Delta v = |v + w|^{p - 1} (v + w), \qquad v(0) = 0,
\end{equation}
and
\begin{equation}\label{4.17}
i w_{t} + \Delta w = 0, \qquad w(0) = u_{0},
\end{equation}
on $[0, \infty)$. Let $\mathcal E(t)$ denote the pseudoconformal energy of $v$,
\begin{equation}\label{4.18}
\mathcal E(t) = \| (x + 2it \nabla) v \|_{L^{2}}^{2} + \frac{8}{p + 1} t^{2} \| v \|_{L^{p + 1}}^{p + 1} = \| xv \|_{L^{2}}^{2} + 2 \langle xv, 2it \nabla v \rangle + 8t^{2} E(t).
\end{equation}

\begin{lemma}\label{l4.2}
If $u_{0} \in B_{1,1}^{\frac{d}{2} + s_{c}}$ is radially symmetric, and $(\ref{1.1})$ has a local solution satisfying $(\ref{4.1})$, then $\mathcal E(1) \lesssim 1$ for $\delta(\| u_{0} \|_{B_{1,1}^{\frac{d}{2} + s_{c}}}) > 0$ sufficiently small.
\end{lemma}
\begin{proof}
Observe that the proof of Lemma $\ref{l4.1}$ also implies
\begin{equation}\label{4.19}
\| \nabla v(1) \|_{L^{2} + L^{\frac{2d}{d - 2}}} \lesssim  \epsilon \cdot \sup_{j < 0} 2^{j \frac{1 - s_{c}}{2}}  \| \nabla u \|_{L_{t}^{2} L_{x}^{\frac{2d}{d - 2}}([2^{j}, 2^{j + 1}] \times \mathbb{R}^{d})}.
\end{equation}
Interpolating $(\ref{4.19})$ with the bound
\begin{equation}\label{4.20}
\| v(1) \|_{\dot{H}^{s_{c}}} \lesssim \| u_{0} \|_{B_{1,1}^{\frac{d}{2} + s_{c}}},
\end{equation}
implies $\| v(1) \|_{L^{p + 1}} \lesssim 1$ for $\delta > 0$ sufficiently small, since $\epsilon = \epsilon(\| u_{0} \|_{B_{1,1}^{\frac{d}{2} + s_{c}}}, p, d, \delta)$.

Using the computations in $(\ref{2.11})$,
\begin{equation}\label{4.21}
\| (x + 2i 1 \nabla) \int_{0}^{1} e^{i(1 - \tau) \Delta} |u|^{p - 1} u d\tau \|_{L_{x}^{2}} \lesssim \| x |u|^{p - 1} u \|_{L_{t}^{1} L_{x}^{2}([0, 1] \times \mathbb{R}^{d})} + \| t \nabla u \|_{L_{t}^{2} L_{x}^{\frac{2d}{d - 2}}([0, 1] \times \mathbb{R}^{d})} \| u \|_{L_{t,x}^{\frac{(d + 2)(p - 1)}{2}}([0, 1] \times \mathbb{R}^{d})}^{p - 1}.
\end{equation}
Then by Lemma $\ref{l4.1}$,
\begin{equation}\label{4.22}
\| t \nabla u \|_{L_{t}^{2} L_{x}^{\frac{2d}{d - 2}}([0, 1] \times \mathbb{R}^{d})} \| u \|_{L_{t,x}^{\frac{(d + 2)(p - 1)}{2}}([0, 1] \times \mathbb{R}^{d})}^{p - 1} \lesssim \delta^{p - 1} \cdot \sup_{j < 0} 2^{j \frac{1 - s_{c}}{2}}  \| \nabla u \|_{L_{t}^{2} L_{x}^{\frac{2d}{d - 2}}([2^{j}, 2^{j + 1}] \times \mathbb{R}^{d})}.
\end{equation}

To handle the first term in $(\ref{4.21})$, consider the cases $\frac{1}{2} \leq s_{c} < 1$ and $0 < s_{c} < \frac{1}{2}$ separately. When $\frac{1}{2} \leq s_{c} < 1$, the radial Sobolev embedding theorem implies
\begin{equation}\label{4.23}
\| x |u|^{p - 1} u \|_{L_{t}^{1} L_{x}^{2}} \lesssim \| x^{\frac{2}{p - 1}} u \|_{L_{t,x}^{\infty}}^{\frac{p - 1}{2}} \| u \|_{L_{t}^{\infty} L_{x}^{\frac{d}{2}(p - 1)}}^{\frac{p - 1}{2} + (1 - c)} \| u \|_{L_{t,x}^{\frac{(d + 2)(p - 1)}{2}}}^{c} \lesssim \| u_{0} \|_{B_{1,1}^{\frac{d}{2} + s_{c}}}^{p - c} \delta^{c},
\end{equation}
where $c \searrow 0$ as $s_{c} \nearrow 1$.

When $0 < s_{c} < \frac{1}{2}$, using the radial Strichartz estimates,
\begin{equation}\label{4.24}
\| x |u|^{p - 1} u \|_{L_{t}^{1} L_{x}^{2}([0, 1] \times \mathbb{R}^{d})} \lesssim \| x^{\frac{d - 1}{2}} u \|_{L_{t}^{\frac{2}{\frac{1}{2} - s_{c}}} L_{x}^{\infty}}^{\frac{2}{d - 1}} \| u \|_{L_{t}^{\infty} L_{x}^{\frac{d}{2}(p - 1)}}^{p - 1 + (1 - \frac{2}{d - 1}) - c} \| u \|_{L_{t,x}^{\frac{(d + 2)(p - 1)}{2}}}^{c} \lesssim \| u_{0} \|_{B_{1,1}^{\frac{d}{2} + s_{c}}}^{p - c} \delta^{c},
\end{equation}
where $c > 0$ for all $0 < s_{c} < \frac{1}{2}$ and $d \geq 3$, with appropriate $p$.

This proves the Lemma.
\end{proof}

\section{Scattering for $(\ref{1.1})$ when $1 < p < 3$ and $0 < s_{c} < 1$}
Having obtained good bounds on the interval $[0, 1]$, we can use the pseudoconformal conservation of energy to extend these bounds to $[1, \infty)$.
\begin{theorem}\label{t6.1}
The initial value problem
\begin{equation}\label{6.1}
i u_{t} + \Delta u = |u|^{p - 1} u, \qquad u(0,x) = u_{0} \in B_{1,1}^{\frac{d}{2} + s_{c}}(\mathbb{R}^{d}), \qquad u : \mathbb{R} \times \mathbb{R}^{d} \rightarrow \mathbb{C},
\end{equation}
is globally well-posed and scattering when $u_{0}$ is radially symmetric. Moreover,
\begin{equation}\label{6.2}
\| u \|_{L_{t}^{\frac{p + 1}{1 - s_{c}}} L_{x}^{p + 1}(\mathbb{R} \times \mathbb{R}^{d})} \leq C (1 + \| u_{0} \|_{B_{1,1}^{\frac{d}{2} + s_{c}}}^{r}),
\end{equation}
for some $C$ that does not depend on $\| u_{0} \|_{B_{1,1}^{\frac{d}{2} + s_{c}}}$ and $r < \infty$.
\end{theorem}
\begin{remark}
When $\| u_{0} \|_{B_{1,1}^{\frac{d}{2} + s_{c}}}$ is small,
\begin{equation}\label{6.3}
\| u \|_{L_{t}^{\frac{p + 1}{1 - s_{c}}} L_{x}^{p + 1}(\mathbb{R} \times \mathbb{R}^{d})} \lesssim \| u_{0} \|_{\dot{H}^{s_{c}}} \lesssim \| u_{0} \|_{B_{1,1}^{\frac{d}{2} + s_{c}}}.
\end{equation}
So it suffices to consider $\| u_{0} \|_{B_{1,1}^{\frac{d}{2} + s_{c}}} \gtrsim 1$.
\end{remark}
\begin{proof}
If $v$ solves $(\ref{4.16})$ on $[1, \infty)$ with $w = 0$, $0 < s_{c} < 1$, and $\mathcal E(1) < \infty$, where $\mathcal E(t)$ is given by $(\ref{4.18})$, then by direct computation,
\begin{equation}\label{6.4}
\frac{d}{dt} \mathcal E(t) = -\frac{4}{p + 1} t \| v \|_{L^{p + 1}}^{p + 1} < 0,
\end{equation}
which implies scattering.

Now compute $\frac{d}{dt} \mathcal E(t)$ when $w$ need not be zero, but $w$ solves $(\ref{4.17})$ with $u_{0} \in B_{1,1}^{\frac{d}{2} + s_{c}}$, radially symmetric. Then by direct computation,
\begin{equation}\label{6.5}
\aligned
\frac{d}{dt} \mathcal E(t) = -\frac{4}{p + 1} t \| v \|_{L^{p + 1}}^{p + 1} - 2 \langle (x + 2it \nabla) v, i (x + 2it \nabla)(|v + w|^{p - 1} (v + w) - |v|^{p - 1} v) \rangle \\ - 8t^{2} \langle |v|^{p - 1} v, i (|v + w|^{p - 1} (v + w) - |v|^{p - 1} v) \rangle.
\endaligned
\end{equation}

When $\frac{1}{2} \leq s_{c} < 1$, by the radial Sobolev embedding theorem, since $\frac{p - 1}{2} < 1$,
\begin{equation}\label{6.6}
\aligned
- 2 \langle (x + 2it \nabla) v, i x (|v + w|^{p - 1} (v + w) - |v|^{p - 1} v) \rangle \\ \lesssim \| (x + 2it \nabla) v \|_{L^{2}} \| x^{\frac{2}{p - 1}} w \|_{L^{\infty}}^{\frac{p - 1}{2}} \| w \|_{L^{p + 1}}^{1 - \frac{p - 1}{2}} (\| v \|_{L^{p + 1}}^{p - 1} + \| w \|_{L^{p + 1}}^{p - 1}) \\
\lesssim \mathcal E(t)^{1/2} \| u_{0} \|_{B_{1,1}^{\frac{d}{2} + s_{c}}}^{\frac{p - 1}{2}} \| w \|_{L^{p + 1}}^{1 - \frac{p - 1}{2}} (\| v \|_{L^{p + 1}}^{p - 1} + \| w \|_{L^{p + 1}}^{p - 1}).
\endaligned
\end{equation}
When $0 < s_{c} < \frac{1}{2}$, split
\begin{equation}\label{6.7}
xw = (x + 2it \nabla) w - 2it \nabla w.
\end{equation}
Again by $(\ref{2.12})$ and the radial Sobolev embedding theorem, for $s_{c} < \frac{d}{2} - 1$, $\| y u_{0} \|_{\dot{H}^{s_{c} + 1}} \lesssim \| u_{0} \|_{\dot{H}^{s_{c}}}$, so interpolating the Strichartz estimate,
\begin{equation}\label{6.8}
\| e^{it \Delta} u_{0} \|_{L_{t}^{\frac{p + 1}{1 - s_{c}}} L_{x}^{p + 1}} \lesssim \| u_{0} \|_{\dot{H}^{s_{c}}},
\end{equation}
with the Littlewood--Paley projection estimate
\begin{equation}\label{6.9}
\| P_{j} e^{it \Delta} u_{0} \|_{L_{t,x}^{\infty}} \lesssim \| P_{j} u_{0} \|_{\dot{H}^{d/2}},
\end{equation}
implies that
\begin{equation}\label{6.10}
\| (x + 2it \nabla) w \|_{L_{t}^{\frac{2(p + 1)}{3 - p} \frac{1}{1 - s_{c}}} L_{x}^{\frac{2(p + 1)}{3 - p}}} \lesssim \| u_{0} \|_{B_{1,1}^{\frac{d}{2} + s_{c}}}.
\end{equation}
Finally, since $\frac{2}{p - 1} < \frac{d}{2}$, by the Sobolev embedding theorem, for $\frac{1}{q} = s_{c}$, by dispersive estimates and $s_{c} = \frac{d}{2} - \frac{2}{p - 1}$,
\begin{equation}\label{6.11}
\| |\nabla|^{\frac{2}{p - 1}} e^{it \Delta} u_{0} \|_{L_{x}^{\infty}} \lesssim \| e^{it \Delta} u_{0} \|_{B_{1, q}^{\frac{d}{2}}} \lesssim \frac{1}{t^{\frac{2}{p - 1}}} \| u_{0} \|_{B_{1,q'}^{\frac{d}{2}}} \lesssim \frac{1}{t^{\frac{2}{p - 1}}} \| u_{0} \|_{B_{1,1}^{\frac{d}{2} + s_{c}}}.
\end{equation}
Therefore,
\begin{equation}\label{6.12}
\| t^{\frac{2}{p - 1}} |\nabla|^{\frac{2}{p - 1}} w \|_{L^{\infty}}^{\frac{p - 1}{2}} \lesssim \| u_{0} \|_{B_{1,1}^{\frac{d}{2} + s_{c}}}^{\frac{p - 1}{2}}.
\end{equation}
Therefore, we have proved
\begin{equation}\label{6.13}
\| xw \|_{L_{t}^{\frac{2(p + 1)}{3 - p} \frac{1}{1 - s_{c}}} L_{x}^{\frac{2(p + 1)}{3 - p}}} \lesssim \| u_{0} \|_{B_{1,1}^{\frac{d}{2} + s_{c}}}.
\end{equation}

Next, integrating by parts,
\begin{equation}\label{6.14}
\aligned
- 2 \langle (2it \nabla) v, i (2it \nabla)(|v + w|^{p - 1} (v + w) - |v|^{p - 1} v) \rangle = -8t^{2} \langle \nabla v, i \nabla (|v + w|^{p - 1} (v + w) - |v|^{p - 1} v) \rangle \\
= 8t^{2} \langle \Delta v, i(|v + w|^{p - 1} (v + w) - |v|^{p - 1} v) \rangle.
\endaligned
\end{equation}
Summing, by $(\ref{4.16})$,
\begin{equation}\label{6.15}
\aligned
8t^{2} \langle \Delta v, i(|v + w|^{p - 1} (v + w) - |v|^{p - 1} v) \rangle - 8t^{2} \langle |v|^{p - 1} v, i (|v + w|^{p - 1} (v + w) - |v|^{p - 1} v) \rangle \\
= 8t^{2} \langle -i v_{t},  i(|v + w|^{p - 1} (v + w) - |v|^{p - 1} v) \rangle \\ - 8t^{2} \langle  (|v + w|^{p - 1} (v + w) - |v|^{p - 1} v), i (|v + w|^{p - 1} (v + w) - |v|^{p - 1} v) \rangle \\
= -8t^{2} \langle v_{t}, (|v + w|^{p - 1} (v + w) - |v|^{p - 1} v) \rangle.
\endaligned
\end{equation}
Finally, integrating by parts,
\begin{equation}\label{6.16}
\aligned
- 2 \langle xv, i (2it \nabla)(|v + w|^{p - 1} (v + w) - |v|^{p - 1} v) \rangle = 4t \langle xv, \nabla(|v + w|^{p - 1} (v + w) - |v|^{p - 1} v) \rangle \\
= -4td \langle v, (|v + w|^{p - 1} (v + w) - |v|^{p - 1} v) \rangle -4t \langle x \cdot \nabla v, (|v + w|^{p - 1}(v + w) - |v|^{p - 1} v) \rangle.
\endaligned
\end{equation}
Integrating by parts,
\begin{equation}\label{6.17}
-4t \langle x \cdot \nabla v, (|v + w|^{p - 1} (v + w) - |v|^{p - 1} v) \rangle = \frac{4dt}{p + 1} (\| v + w \|_{L^{p + 1}}^{p + 1} - \| v \|_{L^{p + 1}}^{p + 1}) + 4t \langle x \cdot \nabla w, |v + w|^{p - 1} (v + w) \rangle.
\end{equation}
Now then, summing,
\begin{equation}\label{6.18}
4t \langle x \cdot \nabla w, |v + w|^{p - 1} (v + w)  \rangle = 4t \langle (x + 2it \nabla) \cdot \nabla w, |v + w|^{p - 1} (v + w) \rangle - 8t^{2} \langle i \Delta w, |v + w|^{p - 1} (v + w) \rangle.
\end{equation}
Summing $(\ref{6.15})$ and $(\ref{6.18})$, since $w_{t} = i \Delta w$,
\begin{equation}\label{6.19}
(\ref{6.15}) + (\ref{6.18}) = 4t \langle (x + 2it \nabla) \cdot \nabla w, |v + w|^{p - 1} (v + w) \rangle - 8t^{2} \langle v_{t} + w_{t}, |v + w|^{p - 1} (v + w) \rangle + 8t^{2} \langle v_{t}, |v|^{p - 1} v \rangle.
\end{equation}
By the radial Sobolev embedding theorem, $(\ref{2.12})$, and the fact that $\frac{d}{2} > 1$,
\begin{equation}\label{6.20}
\| (x + 2it \nabla) \cdot \nabla w \|_{L_{t}^{\frac{p + 1}{1 - s_{c}}} L_{x}^{p + 1}} \lesssim \| u_{0} \|_{B_{1,1}^{\frac{d}{2} + s_{c}}}.
\end{equation}

Therefore, plugging these computations back into $(\ref{6.5})$,
\begin{equation}\label{6.21}
\aligned
\frac{d}{dt} \mathcal E(t) \lesssim -\frac{4}{p + 1} t \| v \|_{L^{p + 1}}^{p + 1} - \frac{8t^{2}}{p + 1} \frac{d}{dt} \| v + w \|_{L^{p + 1}}^{p + 1} + \frac{8t^{2}}{p + 1} \frac{d}{dt} \| v \|_{L^{p + 1}}^{p + 1} + t \| (x + 2it \nabla) \cdot \nabla w \|_{L^{p + 1}} \| v + w \|_{L^{p + 1}}^{p} \\
+ t \| v \|_{L^{p + 1}} \| w \|_{L^{p + 1}} (\| v \|_{L^{p + 1}}^{p - 1} + \| w \|_{L^{p + 1}}^{p - 1}) + t \| w \|_{L^{p + 1}}^{p + 1} + \mathcal E(t)^{1/2} \| (x + 2it \nabla)w \|_{L^{\frac{2(p + 1)}{3 - p}}} (\| v \|_{L^{p + 1}}^{p - 1} + \| w \|_{L^{p + 1}}^{p - 1}) \\
+ \mathcal E(t)^{1/2} \| u_{0} \|_{B_{1,1}^{\frac{d}{2} + s_{c}}}^{\frac{p - 1}{2}} \| w \|_{L^{p + 1}}^{1 - \frac{p - 1}{2}} (\| v \|_{L^{p + 1}}^{p - 1} + \| w \|_{L^{p + 1}}^{p - 1}).
\endaligned
\end{equation}
Then by the product rule,
\begin{equation}\label{6.22}
\aligned
\frac{d}{dt} [\mathcal E(t) + \frac{8t^{2}}{p + 1} \| v + w \|_{L^{p + 1}}^{p + 1} - \frac{8 t^{2}}{p + 1} \| v \|_{L^{p + 1}}] \\ \lesssim -\frac{4}{p + 1} t \| v \|_{L^{p + 1}}^{p + 1} + t \| (x + 2it \nabla) \cdot \nabla w \|_{L^{p + 1}} \| v + w \|_{L^{p + 1}}^{p}
+ t \| v \|_{L^{p + 1}}^{p} \| w \|_{L^{p + 1}} + t \| w \|_{L^{p + 1}}^{p + 1} \\ + \mathcal E(t)^{1/2} \| (x + 2it \nabla) w \|_{L^{\frac{2(p + 1)}{3 - p}}} (\| v \|_{L^{p + 1}}^{p - 1} + \| w \|_{L^{p + 1}}^{p - 1})
+ \mathcal E(t)^{1/2} \| u_{0} \|_{B_{1,1}^{\frac{d}{2} + s_{c}}}^{\frac{p - 1}{2}} \| w \|_{L^{p + 1}}^{1 - \frac{p - 1}{2}} (\| v \|_{L^{p + 1}}^{p - 1} + \| w \|_{L^{p + 1}}^{p - 1}).
\endaligned
\end{equation}

Since $\| v(1) \|_{L^{p + 1}} \lesssim 1$, $\mathcal E(1) \lesssim 1$, and dispersive estimates imply that $\| w(1) \|_{L^{p + 1}} \lesssim 1$, the Cauchy--Schwartz inequality and $(\ref{6.22})$ imply that
\begin{equation}\label{6.23}
\aligned
\frac{1}{t^{2}} \mathcal E(t) \lesssim \frac{1}{t^{2}} + \| w(t) \|_{L^{p + 1}} (\| v(t) \|_{L^{p + 1}}^{p} + \| w(t) \|_{L^{p + 1}}^{p}) + \frac{1}{t^{2}} \int_{1}^{t}  \tau \| (x + 2it \nabla) \cdot \nabla w \|_{L^{p + 1}}^{p + 1} d\tau \\
+ \frac{1}{t^{2}} \int_{1}^{t} \tau \| w \|_{L^{p + 1}}^{p + 1} + \frac{1}{t^{2}} \int_{1}^{t} \frac{\mathcal E(t)^{1/2}}{\tau} \cdot \tau  \| (x + 2it \nabla)w \|_{L^{\frac{2(p + 1)}{3 - p}}} (\| v \|_{L^{p + 1}}^{p - 1} + \| w \|_{L^{p + 1}}^{p - 1}) d\tau \\
+ \frac{1}{t^{2}} \int_{1}^{t} \frac{\mathcal E(\tau)^{1/2}}{\tau} \cdot \tau \| u_{0} \|_{B_{1,1}^{\frac{d}{2} + s_{c}}}^{\frac{p - 1}{2}} \| w \|_{L^{p + 1}}^{1 - \frac{p - 1}{2}} (\| v \|_{L^{p + 1}}^{p - 1} + \| w \|_{L^{p + 1}}^{p - 1}) d\tau,
\endaligned
\end{equation}
with implicit constants depending only on $p$ and $d$. Then choosing $0 < \delta(p, d) \ll 1$ sufficiently small, by the Cauchy--Schwartz inequality,
\begin{equation}\label{6.24}
\aligned
\frac{1}{t^{2}} \mathcal E(t) \lesssim \frac{1}{t^{2}} + \| w(t) \|_{L^{p + 1}} (\| v(t) \|_{L^{p + 1}}^{p} + \| w(t) \|_{L^{p + 1}}^{p}) + \frac{1}{t^{2}} \int_{1}^{t}  \tau \| (x + 2it \nabla) \cdot \nabla w \|_{L^{p + 1}}^{p + 1} d\tau \\
+ \frac{1}{t^{2}} \int_{1}^{t} \tau \| w \|_{L^{p + 1}}^{p + 1} + \frac{\delta}{t^{2}} \int_{1}^{t} \frac{\mathcal E(t)}{\tau^{2}} \tau d\tau + \frac{1}{\delta t^{2}} \int_{1}^{t} \tau  \| (x + 2it \nabla) w \|_{L^{\frac{2(p + 1)}{3 - p}}}^{2} (\| v \|_{L^{p + 1}}^{2(p - 1)} + \| w \|_{L^{p + 1}}^{2(p - 1)}) d\tau \\
+ \frac{1}{\delta t^{2}} \int_{1}^{t} \tau \| u_{0} \|_{B_{1,1}^{\frac{d}{2} + s_{c}}}^{p - 1} \| w \|_{L^{p + 1}}^{3 - p} (\| v \|_{L^{p + 1}}^{2(p - 1)} + \| w \|_{L^{p + 1}}^{2(p - 1)}) d\tau.
\endaligned
\end{equation}

Therefore, by Young's inequality,
\begin{equation}\label{6.25}
\aligned
\| \frac{1}{t^{2}} \mathcal E(t) \|_{L_{t}^{\frac{1}{1 - s_{c}}}([1, \infty))} \lesssim 1 + \| \| v(t) \|_{L^{p + 1}}^{p + 1} \|_{L_{t}^{\frac{1}{1 - s_{c}}}([1, \infty))}^{\frac{p}{p + 1}} \| \| w(t) \|_{L^{p + 1}}^{p + 1} \|_{L_{t}^{\frac{1}{1 - s_{c}}}([1, \infty))}^{\frac{1}{p + 1}} + \| \| w(t) \|_{L^{p + 1}}^{p + 1} \|_{L_{t}^{\frac{1}{1 - s_{c}}}([1, \infty))} \\
 + \| \| (x + 2it \nabla) \cdot \nabla w \|_{L^{p + 1}}^{p + 1} \|_{L_{t}^{\frac{1}{1 - s_{c}}}([1, \infty))} + \frac{1}{\delta} \| u_{0} \|_{B_{1,1}^{\frac{d}{2} + s_{c}}}^{p - 1} \| \| w(t) \|_{L^{p + 1}}^{p + 1} \|_{L_{t}^{\frac{1}{1 - s_{c}}}([1, \infty))} \\ + \frac{1}{\delta} \| u_{0} \|_{B_{1,1}^{\frac{d}{2} + s_{c}}}^{p - 1} \| \| w(t) \|_{L^{p + 1}}^{p + 1} \|_{L_{t}^{\frac{1}{1 - s_{c}}}([1, \infty))}^{\frac{3 - p}{p + 1}} \| \| v(t) \|_{L^{p + 1}}^{p + 1} \|_{L_{t}^{\frac{1}{1 - s_{c}}}([1, \infty))}^{\frac{2(p - 1)}{p + 1}} \\
+ \frac{1}{\delta} \| (x + 2it \nabla) w \|_{L_{t}^{\frac{2(p + 1)}{3 - p} \frac{1}{1 - s_{c}}} L_{x}^{\frac{2(p + 1)}{3 - p}}}^{2} \| \| v(t) \|_{L^{p + 1}}^{p + 1} \|_{L_{t}^{\frac{1}{1 - s_{c}}}([1, \infty))}^{\frac{2(p - 1)}{p + 1}} \\ + \frac{1}{\delta} \| (x + 2it \nabla) w \|_{L_{t}^{\frac{2(p + 1)}{3 - p} \frac{1}{1 - s_{c}}} L_{x}^{\frac{2(p + 1)}{3 - p}}}^{2} \| \| w(t) \|_{L^{p + 1}}^{p + 1} \|_{L_{t}^{\frac{1}{1 - s_{c}}}([1, \infty))}^{\frac{2(p - 1)}{p + 1}}.
\endaligned
\end{equation}
Then combining $\| v(t) \|_{L^{p + 1}}^{p + 1} \lesssim \frac{1}{t^{2}} \mathcal E(t)$, Strichartz estimates, $(\ref{6.6})$--$(\ref{6.13})$, and $(\ref{6.20})$,
\begin{equation}\label{6.26}
\| \frac{1}{t^{2}} \mathcal E(t) \|_{L_{t}^{\frac{1}{1 - s_{c}}}([1, \infty))} \lesssim_{p, d} 1 + \| u_{0} \|_{B_{1,1}^{\frac{d}{2} + s_{c}}}^{p + 1} + \| u_{0} \|_{B_{1,1}^{\frac{d}{2} + s_{c}}}^{2p} + \| u_{0} \|_{B_{1,1}^{\frac{d}{2} + s_{c}}}^{\frac{2(p + 1)}{3 - p}}.
\end{equation}
This proves the theorem.
\end{proof}

\section*{Acknowledgements}
The author was supported by NSF grant DMS-$1764358$ during the writing of this paper. The author is grateful to Frank Merle for many helpful conversations regarding this problem.

\bibliography{biblio}
\bibliographystyle{plain}
\medskip

\end{document}